\documentclass[a4paper,12pt]{amsart}
\usepackage{amsfonts}
\usepackage{amsthm}
\usepackage{amscd}
\usepackage{amssymb}

\numberwithin{equation}{section}
\newtheorem{theorem}[equation]{Theorem}

\newtheorem{lemma}[equation]{Lemma}
\newtheorem{proposition}[equation]{Proposition}
\newtheorem{conjecture}[equation]{Conjecture}

\newtheorem*{nil-theorem}{Nil-Theorem}

\newtheorem{corollary}[equation]{Corollary}

\newtheorem{example}[equation]{Example}

\newtheorem{remark}[equation]{Remark}

\title[Non-symplectic smooth circle actions]
{\bf Non-symplectic smooth circle actions \\ on symplectic manifolds}

\author{Bogus{\l}aw Hajduk, Krzysztof Pawa{\l}owski,\\
        and Aleksy Tralle}

\begin{document}
\maketitle

\begin{abstract}
We present some methods to construct smooth circle actions on symplectic manifolds 
with non-symplectic fixed point sets or non-symplectic cyclic isotropy point sets. 
All such actions are not compatible with any symplectic form. To cover the case 
of non-symplectic fixed point sets, we use non-symplectic $4$-manifolds which
become symplectic after taking the product of the manifold and the $2$-sphere.
In turn, the case of non-symplectic cyclic isotropy point set is obtained by
constructing smooth circle actions on spheres with non-symplectic cyclic 
isotropy point sets, and then by taking the equivariant connected sum of the sphere 
and the cartesian product of copies of the $2$-sphere. By using mapping tori, 
we can convert smooth cyclic group actions into non-symplectic smooth fixed point free 
circle actions on symplectic manifolds.

\bigskip
\small{


\noindent
{\bf AMS classification (2000)}: Primary 53D05; Secondary 57S25.}

\end{abstract}

\section{Introduction}
One of the intriguing questions on symplectic manifolds is to give,
for a closed manifold $X$, sufficient and necessary conditions on
$M=X\times S^1$ to admit a symplectic structure. If $M$ admits a
symplectic form $\omega$ invariant with respect to the obvious
action of the circle, then on $X$ we have a non-vanishing and closed
1-form $\iota_V\omega ,$ where $V$ is the vector field generating
the action. This implies that $X$ fibres over a circle.  Conversely,
if $X$ admits a symplectic fibration over the circle, then
$X\times S^1$ admits a symplectic structure (compare the proof of
Proposition~\ref{pro:torus}). No other examples are known. In dimension 4, 
an answer to the described problem is known.

\begin{theorem}\label{friedl} \cite{FV} If $X$ is a closed
irreducible 3-manifold with vanishing Thurston norm, then $X\times S^1$
is symplectic if and only if $X$ fibres over a circle.
\end{theorem}

A more general conjecture was stated by Scott Baldridge in \cite{B}.

\begin{conjecture}
Every closed $4$-manifold that admits a symplectic form and a smooth circle action
also admits a symplectic circle action (with respect to a possibly different symplectic form).
\end{conjecture}

In the same paper \cite{B}, Baldgidge proved that the answer is positive for
circle actions with non-empty fixed point sets.

\begin{theorem}\label{bald}\cite{B} If $M$ is a closed symplectic $4$-manifold with a circle action
such that the fixed point set is non-empty, then there exists a symplectic circle
action on $M$.
\end{theorem}

It seems unlikely that this continue to be true in higher
dimensions. On the other side, symplecticness of the action does
imply some topological restrictions on a manifold. For example, one
can repeat the argument above to show that if a circle action is
free and compatible with a symplectic form, then  $X/S^1$ also
fibers over a circle.  Thus, one can ask the following question:
{\it given a closed symplectic manifold $M$ with a smooth circle
action, when does $M$ admit a symplectic circle action?}

We wish to gain some understanding of these compatibility problems
by looking first at constructions of Lie group actions which are
non-compatible with any  symplectic form. Note that actions which
are non-compatible with a given symplectic form obviously exist.
Simply deform an invariant symplectic form by a non-equivariant
diffeomorphism or deform the action by a non-symplectic
diffeomorphism.   In \cite{A}, Allday constructed examples of
cohomologically symplectic manifolds with circle actions which
cannot be symplectic for cohomological reasons. However, there is no
method in sight to see if manifolds in his examples are symplectic,
since the surgery used in the construction is apparently
non-symplectic. Hence, it is still interesting to look for
constructions of {\it symplectic} manifolds with exotic, from the
point of view of symplectic topology, smooth circle actions.

Yet there are other reasons to study examples along the borderline
between symplectic and topological properties of group actions.
In fact, if a symplectic manifold does admit a symplectic circle action,
there are various restrictions on topology of $M$, e.g. expressed in
cohomological terms (cf. \cite{A}, \cite{Oz}), on characteristic classes
(cf. \cite{Fe}), or on Novikov cohomology (cf. \cite{Fa}).
However, as in \cite{A}, some of these results can be proved by purely cohomological means.

A specific example of this phenomenon are manifolds which are
asymmetric, i.e. do not admit any action of a compact Lie group.
It was discovered some 40 years ago that closed asymmetric manifolds
exist, and there are plenty of them, see Volker Puppe's survey
\cite{Pu}. Symplectic  manifolds with no symplectic action of any
compact connected Lie group are also known. Namely, Polterovich \cite{Po}
has shown that for any closed symplectic manifold
$(M,\omega)$ with $\pi_2(M)=0$ and centerless $\pi_1(M),$ the
identity component $\operatorname{Symp}_0(M,\omega)$ of the
symplectomorphism group is torsion free. However, it seems that
these manifolds do not admit circle actions.

In this paper we construct examples of two types. First, there are
smooth circle actions on closed symplectic manifolds with
non-symplectic sets of fixed points, resulting from an action on
a manifold $X$ with isolated fixed points and two $4$-manifolds
$M$ and $N$ such that $M\times X$ is diffeomorphic to $N\times X$,
while only one of $M$ and $N$ is symplectic.

The other type of examples is obtained from smooth actions of the circle $S^1$ 
on spheres with prescribed fixed point sets and $\mathbb{Z}_{pq}$-isotropy 
point sets for distinct primes $p$ and $q$. By forming the equivariant
connected sum of the sphere and the cartesian product of copies of the $2$-sphere
with some action of $S^1$, we obtain smooth actions of $S^1$ on symplectic manifolds 
with non-symplectic $\mathbb{Z}_{pq}$-isotropy point sets inherited from the spheres.

Once we have a specific smooth action of a cyclic group $\mathbb{Z}_r$ on a symplectic 
manifold $M$ with fixed point set $F$ such that $S^1 \times S^1 \times F$ is not symplectic, 
the mapping torus construction allows us to convert $M$ into symplectic manifold 
with a smooth fixed point free action of $S^1$ such that the $\mathbb{Z}_r$-isotropy 
point set is diffeomorphic to $S^1 \times S^1 \times F$, and therefore 
the constructed action of $S^1$ is not symplectic (see Proposition~\ref{pro:torus}).

\section{Fixed point sets of symplectic actions}

The following is well-known, cf. \cite{GuSt}, Lemma 27.1.

\begin{lemma}\label{sympfix}
Let $G$ be a compact Lie group. If $G$ acts symplectically on a
symplectic manifold $(M,\omega)$, then the fixed point set $M^G$
is a symplectic submanifold.
\end{lemma}

\begin{proof}
Let $x\in M^G.$ Then, with respect to a chosen invariant Riemannian metric,
$G$ acts on a normal slice via a faithful orthogonal representation.
Thus, a vector $U$ of $T_x(M)$ belongs to $T_x(M^G)$ if and only if
$g_*U=U$ for every $g\in G$. Moreover, vectors of the form $V-g_*V$
span a subspace of $T_x(M)$ transversal to $M^G.$ Hence, for $U\in T_x(M^G)$,
we have $\omega (U,V)=\omega (g_*U,g_*V)=\omega (U,g_*V)$, and
therefore $\omega (U,V-g_*V)=0$ for any $g\in G$ and $V\in T_x(M)$.
So, if $\omega (U,W)=0$ for all $W\in T_x(M^G)$, then also $\omega
(U,W')=0$ for all $W'\in T_x(M)$ and this implies that $U = 0$.
Thus $\omega |T(M^G)$ is symplectic.
\end{proof}

\begin{corollary}
Let $G$ be a compact Lie group and let $H$ be a closed subgroup of
$G$. If $G$ acts symplectically on a symplectic manifold $M$, then
the set of points with isotropy equal to $H$ is a symplectic
manifold.
\end{corollary}

An analogous property  for almost complex manifolds and actions is
straightforward.

\begin{lemma} If a compact Lie group $G$ acts smoothly on an almost complex
manifold $M$ preserving an almost complex structure $J$, then the fixed point
set $M^G$ is a J-holomorphic submanifold of $M.$
\end{lemma}

\begin{proof}
If $J$ is $G$-invariant, $g_*(JU)=g_*Jg_*^{-1}g_*U=JU$
for any $U\in T_xM^G$ and $g\in G$.
\end{proof}

\section{Circle actions with non-symplectic fixed point sets}\label{circle}

For closed simply connected $6$-manifolds $M$ whose second Stiefel--Whitney class $w_2(M)$ vanishes,
the diffeomorphism type of $M$ is completely determined by the cohomology ring $M^{*}(M,\mathbb{Z})$
and the first Pontriagin class $p_1(M)$.
More precisely, Theorem 3 in the book of Wall \cite{W} can be stated as follows.

\begin{theorem}\label{thm:wall}
The diffeomorphism classes of closed simply connected $6$-manifolds $M$
with torsion free $H^{*}(M,\mathbb{Z})$ and $w_2(M) = 0$ correspond bijectively
to the isomorphism classes of an algebraic invariant consisting of:
\begin{enumerate}
\item two free abelian groups $H=H^2(M;\mathbb{Z})$ and $G=H^3(M;\mathbb{Z})$,
\item a symmetric trilinear map
$\mu: H\times H\times H\rightarrow \mathbb{Z}$ given by the cup product,
\item a homomorphism $p_1: H\rightarrow\mathbb{Z}$ determined by the
first Pontriagin class $p_1$.
\end{enumerate}
\end{theorem}

In the sequel, for a space $X$, the product of $n$ copies of $X$ is denoted
by $\prod{\!^n}X$, and the disjoint union of $n$ copies of $X$ is denoted
by $\coprod{\!^n}X$.

\begin{proposition}\label{four}
Let $M$ and $N$ be two closed simply connected $4$-dimensional smooth
manifolds such that the following condition holds.
\begin{enumerate}
\item $M$ and $N$ are homeomorphic, but $M$ and $N$ are not diffeomorphic.
\item $M$ is not a symplectic manifold and $N$ admits a symplectic structure.
\item The second Stiefel--Whitney class $w_2(M)$ vanishes, $w_2(M) = 0$.
\item The cohomology ring $H^{*}(M,\mathbb{Z})$ is torsion free.
\end{enumerate}
Then for any $n \geq 1$, the manifold $M \times \prod{\!^n}S^2$ is symplectic
and admits a smooth action of $S^1$ not compatible with any symplectic form.
\end{proposition}

\begin{proof}
It follows from Theorem~\ref{thm:wall} that $M \times S^2$ is {\it diffeomorphic}
to $N\times S^2$. Indeed, under our assumptions,
$w_2(M\times S^2)=w_2(N\times S^2)=0$.
Recall that $p_1(M\times S^2)$ is inherited  from $M,$ and $p_1$ is a
topological invariant for closed 4-manifolds. Therefore $p_1(M\times S^2)
= p_1(N\times S^2)$, and $M\times S^2$ is diffeomorphic to $N\times S^2.$
As the manifold $N\times S^2$ is symplectic, so are $M\times S^2$ and
$M\times \prod{\!^n}S^2$ for any $n \geq 1$.

Now, consider the diagonal action of $S^1$ on $M\times \prod{\!^n}S^2$,
where $S^1$ acts trivially on $M$ and $S^1$ acts smoothly on
$\prod{\!^n}S^2$ with a finite number $k$ of fixed points. Then the
fixed point set of the diagonal action of $S^1$ on $M\times \prod{\!^n}S^2$
is diffeomorphic to $\coprod{\!^k} M$. So, by Lemma \ref{sympfix},
the action is not symplectic with respect to any symplectic structure.
\end{proof}

\begin{example}
{\rm
Some examples of manifolds $M$ and $N$ as required above are
obtained by applying to symplectic $4$-manifolds constructions such as logarithmic
transformation or knot surgery. To detect both non-diffeomorphism
and non-symplecticness one uses  Taubes' theorem that for a symplectic manifold $N$,
the Seiberg--Witten invariant $SW_N$ is equal to $\pm 1$ on a class $u\in H^2(M;\mathbb Z)$.
There are examples of symplectic manifolds $N$ such that it is possible to
obtain from $N$ a smooth manifold $M$ of the same topological type
but with $SW_M(u)\neq \pm 1$ for any $u$ (see Sec. 12.4 of \cite{S}, or \cite{P}).
An explicit example is the Barlow surface which is non-symplectic and
homeomorphic to $\mathbb{C}P^2$ blown up in 8 points. There exists also
a non-symplectic manifold homeomorphic to the K3 surface.}
\end{example}

\section{Circle actions with non-symplectic cyclic isotropy point sets}

For any integer $k \geq 0$, we shall consider the representation $t^k
\colon S^1 \to U(1)= S^1$ given by $t^k(z) = z^k$ for all $z \in S^1$,
and we write $t^k + t^{\ell} \colon S^1 \to U(2)$ to denote the direct
sum of $t^k$ and $t^{\ell}$ for $k, \ell \geq 0$. Moreover, $n t^k$
denotes the direct sum $t^k + \dots + t^k$, $n$-times.

For $G = S^1$ and $H = \mathbb{Z}_{pq}$ for two distinct primes $p$ and $q$,
we wish to give examples of smooth actions of $G$ on symplectic manifolds
$M$ such that the $H$-isotropy point set $M_H = \{x \in M \ | \ G_x = H \}$
is not a symplectic manifold, and thus the action of $G$ on $M$ is not symplectic
with respect to any possible symplectic structure on $M$.

Hereafter, $D^n$ denotes the $n$-dimensional disk. First, we construct a smooth action
of $S^1$ on the disk $D^7$ with prescribed properties.

\begin{theorem}\label{thm:action-disk}
Let $G = S^1$ and $H = \mathbb{Z}_{pq}$ for two distinct primes $p$ and $q$.
Then there exists a smooth action of $G$ on the disk $D^7$ such that
the following conclusions hold.
\begin{enumerate}
\item The family of isotropy subgroups in $D^7$ consists of $G$, $H$,
$\mathbb{Z}_p$, $\mathbb{Z}_q$, and $\{1\}$.
\item The manifold $D^7_G$ is diffeomorphic to $D^1$.
\item The manifold $D^7_H$ is $G$-diffeomorphic to
      $\coprod{\!^k}G/H \cong \coprod{\!^k}S^1$ for any $k \geq 1$.
\item For any $x \in D^7_G$ {\rm (resp. $D^7_H$)}, the representation of
      $G$ {\rm (resp. $H$)} on the normal space to $D^7_G$ {\rm (resp. $D^7_H$)}
      in $D^7$ at $x$ is isomorphic to $V$ {\rm (resp. ${\rm Res}^G_H(V)$)},
      where $V = \mathbb{R}^6$ with the action of $G$ given by the representation
      $t^1 + t^p + t^q$.
\end{enumerate}
\end{theorem}

\begin{proof}
As in the proof of \cite[Lemma~1]{Pa}, consider the action of $G$ on
the disk $D^3$ by identifying $D^3$ with $D^2 \times D^1$, where $G$
acts via the representation $t^1$ on
$D^2$ and trivially on $D^1$. Now, form the quotient space
$Y = D^3/{\sim}$ by the following identifications
on the boundary $S^2 = S^2_{+} \cup S^2_{-}$ of $D^3$:
$$S^2_{+}/{\mathbb{Z}_p}, \ \  S^2_{-}/{\mathbb{Z}_q}, \ \ \hbox{and} \ \ (S^2_{+} \cap S^2_{-})/H.$$
Then $Y$ is a finite contractible $G$-CW complex built up from two
disjoint cells $G/G \times D^1$ and $G/H$, by attaching three cells
$G/{\mathbb{Z}_p} \times D^1$, $G/{\mathbb{Z}_q} \times D^1$, and $G \times D^2$.
In particular,
$$Y^G = D^1 \ \ \hbox{and} \ \ Y_H = G/H \cong S^1.$$

For an integer $k \geq 1$, set $X = Y \cup \dots \cup Y$, the union of $k$-copies
of $Y$ along $Y^G$. Then $X$ is a finite contractible $G$-CW complex such that
$X^H = X^G \sqcup X_H$ with
$$X^G = D^1 \ \ \hbox{and} \ \ X_H = \coprod{\!^k} G/H \cong \coprod{\!^k} S^1.$$
Clearly, the family of isotropy subgroups in $X \setminus X^H$ consists of
$\mathbb{Z}_p$, $\mathbb{Z}_q$, and $\{1\}$.

Let $V = \mathbb{R}^6$ with the action of $G$ given by the representation
$t^1 + t^p + t^q \colon G \to U(3)$. Then the product $G$-vector bundle
$X \times (\mathbb{R} \oplus V)$ over $X$, where $G$ acts trivially on $\mathbb{R}$,
when restricted over $X^G$ and $X_H$, splits as follows:
$$X^G \times (\mathbb{R} \oplus V) \cong T(X^G) \oplus (X^G \times V),$$
$$X_H \times (\mathbb{R} \oplus V) \cong T(X_H) \oplus (X_H \times V).$$
The total space $M = X^H \times D(V)$ of the disk bundle $X^H \times D(V)$ over $X^H$
is a compact smooth $7$-manifold upon which $G$ acts smoothly in such 
a way that 
$$M^G = X^G, \ M_H = X_H, \ \hbox{and} \ M^H = M^G \sqcup M_H = X^H,$$
and the family of isotropy subgroups in $M \setminus M^H$ consists of
$\mathbb{Z}_p$, $\mathbb{Z}_q$, and $\{1\}$. Moreover, for any point $x \in M_G$
{\rm (resp. $M_H$)}, the representation of $G$ (resp. $H$) on the normal space
to $M_G$ {\rm (resp. $M_H$)} in $M$ at $x$ is isomorphic to $V$ {\rm (resp. ${\rm Res}^G_H(V)$)}.

To complete the proof, we can apply the equivariant thickening of
\cite[Proposition~1]{Pa}. This procedure allows us to replace 
the $G$-cells in $X \setminus X^H$,
$$G/{\mathbb{Z}_p} \times D^1, \ \ G/{\mathbb{Z}_q} \times D^1, \ \
  \hbox{and} \ \ G \times D^2,$$
by $G$-handles to convert $M$ into a compact contractible smooth $7$-manifold $D$
equipped with a smooth action of $G$ such that $D \supset M$ as a $G$-invariant
submanifold and the family of isotropy subgroups in $D \setminus M$ consists 
of $\mathbb{Z}_p$, $\mathbb{Z}_q$, and $\{1\}$. As $\partial D$ is simply connected
by the construction, $D$ is diffeomorphic to $D^7$ by the h-Cobordism Theorem.
\end{proof}

\begin{corollary}\label{cor:action-sphere}
Let $G = S^1$ and $H = \mathbb{Z}_{pq}$ for two distinct primes $p$ and $q$.
Then for any given integer $n \geq 1$, there exists a smooth action of
$G$ on the sphere $S^{2n+6}$ such that the following four conclusions hold.
\begin{enumerate}
\item The family of isotropy subgroups in $S^{2n+6}$ consists of $G$, $H$,
      $\mathbb{Z}_p$, $\mathbb{Z}_q$, and $\{1\}$.
\item The manifold $(S^{2n+6})^G = S^{2n+6}_G$ is diffeomorphic to $S^{2n}$.
\item The manifold $S^{2n+6}_H$ is $G$-diffeomorphic to $\coprod{\!^k} (G/H \times S^{2n-1})$
      for any $k \geq 1$.
\item For any $x \in S^{2n+6}_G$ {\rm (resp. $S^{2n+6}_H$)}, the representation of
       $G$ {\rm (resp. $H$)} on the normal space to $S^{2n+6}_G$ {\rm (resp. $S^{2n+6}_H$)}
       in $S^{2n+6}$ at $x$ is isomorphic to $V$ {\rm (resp. ${\rm Res}^G_H(V)$)},
       where $V = \mathbb{R}^6$ with the action of $G$ given by the representation
       $t^1 + t^p + t^q$.
\end{enumerate}
\end{corollary}

\begin{proof}
For a given integer $n \geq 1$, consider $S^{2n+6} = \partial D^{2n+7} \cong 
\partial(D^{2n} \times D^7)$, where $G$ acts trivially on $D^{2n}$, and on $D^7$
as in Theorem~\ref{thm:action-disk}. This yields a smooth action of $G$ on $S^{2n+6}$
such that the conclusions (1)--(4) all hold.
\end{proof}

\begin{remark}
{\rm
In Corollary~\ref{cor:action-sphere}, remove a point from $S^{2n+6}$ fixed under the action
of $G$ to obtain a smooth action of $G$ on $\mathbb{R}^{2n+6}$. As $S^1 \times S^{2n-1}$
is not a symplectic manifold for $n \geq 2$, and $\mathbb{R}^{2n+6}_H$ consists of a number
of copies of $S^1 \times S^{2n-1}$, the action of $G$ on the symplectic manifold
$\mathbb{R}^{2n+6}$ is not symplectic for $n \geq 2$.
}
\end{remark}

\begin{lemma}\label{lem:representation}
Let $G = S^1$ and let $\rho \colon G \to U(n) \subset O(2n)$ be
an orthogonal representation. Then there exists a smooth action of $G$ on
$X^{2n} = S^2 \times \dots \times S^2$, $n$-times, such that for any point $x \in X^{2n}$,
the representation of $G_x$ on $T_x(M)$ is isomorphic ${\rm Res}^G_{G_x}(\rho)$.
\end{lemma}

\begin{proof}
To obtain the required conclusion, consider the diagonal action of $G$ on
$$\mathbb{R}^{2n} = \mathbb{R}^2 \oplus \dots \oplus \mathbb{R}^2, \ n \hbox{-times},$$
obtained from the decomposition of $\rho$ into the irreducible complex
summands and taking their realifications. For every $G$-summand $\mathbb{R}^2$,
the one point compactification
$$\mathbb{R}^2 \cup \{ \infty \} \cong S^2$$
has the obvious action of $G$, either trivial if $G$ acts trivially on $\mathbb{R}^2$,
or with exactly two fixed points, otherwise. The diagonal action of $G$ on $X^{2n}$
has the required property.
\end{proof}

We refer to the action of $S^1$ on the symplectic manifold $X^{2n}$
described in the proof of
Lemma~\ref{lem:representation} as to the {\emph{$\rho$-associated action}
of $S^1$ on $X^{2n}$.

Let $M$ and $N$ be two smooth $G$-manifolds. Assume that for some points
$x \in M^G$ and $y \in N^G$, the representations of $G$ on $T_x(M)$ and
$T_y(N)$
are isomorphic. This allows us to choose some closed disks in $M$ and
$N$ centered
at $x$ and $y$ with equivalent linear actions of $G$. By removing their
interiors
and gluing together their boundaries, we can form the connected sum
$M \#_{x,y} N$ which admits a smooth action of $G$ such that
$$(M \#_{x,y} N)^G \cong M^G \#_{x,y} N^G.$$

We wish to obtain nonsymplectic smooth actions of $S^1$ on
\emph{closed} symplectic manifolds. As in Lemma~\ref{lem:representation},
for $n \geq 1$, let $X^{2n}$ denote the product $S^2 \times \dots \times S^2$, $n$-times.

\begin{theorem}\label{thm:isotropy}
Let $G = S^1$ and $H = \mathbb{Z}_{pq}$ for two distinct primes $p$ and $q$.
Then for any integer $n \geq 1$, there exists a smooth action
of $G$ on $X^{2n+6}$ such that $X^{2n+6}_G = X^{2n}$ and
$$X^{2n+6}_H \cong_G \coprod{\!^k} (G/H \times S^{2n-1})
               \cong \coprod{\!^k} (S^1 \times S^{2n-1})$$
for any given integer $k \geq 1$. If $n \geq 2$, the product $S^1 \times S^{2n-1}$
does not admit a symplectic structure, and thus the action of $G$
on the symplectic manifold $X^{2n+6}$ is not symplectic with respect
to any possible symplectic structure on $X^{2n+6}$.
\end{theorem}

\begin{proof}
According to Corollary~\ref{cor:action-sphere}, for any integer $n \geq 1$,
there exists a smooth action of $G$ on $S^{2n+6}$ such that the following
four conclusions hold.
\begin{enumerate}
\item The family of isotropy subgroups in $S^{2n+6}$ consists of $G$, $H$,
      $\mathbb{Z}_p$, $\mathbb{Z}_q$, and $\{1\}$.
\item The manifold $(S^{2n+6})^G = S^{2n+6}_G$ is diffeomorphic to $S^{2n}$.
\item The manifold $S^{2n+6}_H$ is $G$-diffeomorphic to $\coprod{\!^k} (G/H \times S^{2n-1})$
      for any $k \geq 1$.
\item For any $x \in S^{2n+6}_G$ {\rm (resp. $S^{2n+6}_H$)}, the representation of
       $G$ (resp. $H$) on the normal space to $S^{2n+6}_G$ (resp. $S^{2n+6}_H$) in $S^{2n+6}$
       at $x$ is isomorphic to $V$ {\rm (resp. ${\rm Res}^G_H(V)$)},
       where $V = \mathbb{R}^6$ with the action of $G$ given by the representation
       $t^1 + t^p + t^q$.
  \end{enumerate}
Choose a point $y \in S^{2n+6}_G$ and consider the representation $\rho$ of $G$
on $T_y(S^{2n+6})$. Clearly, $\rho$ is isomorphic to the realification of $n t^0 + t^1 + t^p + t^q$.
Consider the $\rho$-associated action of $G$ on $X^{2n+6}$ (cf. Lemma~\ref{lem:representation}).
In particular, for any $x \in X^{2n+6}_G$, the representation of $G$ on $T_x(X^{2n+6})$
is isomorphic to $\rho$. Form the $G$-equivariant connected sum
$$Y^{2n+6} := X^{2n+6} \#_{x,y} S^{2n+6}$$
and note that
$$Y^{2n+6}_G \cong X^{2n+6}_G  \#_{x,y} S^{2n+6}_G \cong X^{2n} \#_{x,y} S^{2n} \cong X^{2n}.$$
As $X^{2n+6}_H = \varnothing$, it follows that
$$Y^{2n+6}_H = S^{2n+6}_H \cong_G \coprod{\!^k} (G/H \times S^{2n-1})
                            \cong \coprod{\!^k} (S^1 \times S^{2n-1}).$$
Now, the conclusion that $Y^{2n+6}$ is diffeomorphic to $X^{2n+6}$ yields
the required action of $G$ on $X^{2n+6}$.
\end{proof}

\begin{corollary}\label{cor:isotropy}
Let $G = S^1$ and $H = \mathbb{Z}_{pq}$ for two distinct primes $p$ and $q$.
Then for any integer $n \geq 1$, there exists a smooth action
of $G$ on the manifold $M^{2n+8} = T^2 \times X^{2n+6}$ such that
$M^{2n+8}_G = \varnothing$ and
$$M^{2n+8}_H \cong (T^2 \times X^{2n}) \sqcup \coprod{\!^k} (T^3 \times S^{2n-1})$$
for any given integer $k \geq 1$. If $n \geq 2$, the product $T^3 \times S^{2n-1}$
does not admit a symplectic structure, and thus the action of $G$
on the symplectic manifold $M^{2n+8}$ is not symplectic with respect
to any possible symplectic structure on $M^{2n+8}$.
\end{corollary}

\begin{proof}
In order to obtain the required action of $G$ on $M^{2n+8}$,
consider the action of $G$ on $X^{2n+6}$ as in the conclusion of Theorem~\ref{thm:isotropy},
and then take the diagonal action of $G$ on $G/H \times G/H \times X^{2n+6} \cong T^2 \times X^{2n+6}$,
where $G$ acts on $G/H$ in the obvious way.
\end{proof}

Now, we wish to discuss some procedures which allow us to obtain smooth fixed point free
actions of $S^1$ on symplectic manifolds from smooth actions of $\mathbb{Z}_r$ on manifolds
$M$ with non-empty fixed point set $F$.

\begin{proposition}\label{pro:twisted}
Let $M$ be a smooth manifold such that $S^1 \times M$ is a symplectic manifold.
Assume a cyclic group $\mathbb{Z}_r$ acts smoothly on $M$ with non-empty fixed
point set $F$, and for the generator $g$ of $\mathbb{Z}_r$,
the diffeomorphism
$$g \colon M \to M, \ x \mapsto gx$$
is isotopic to the identity on $M$. Then there exists a smooth fixed point free 
action of $S^1$ on $S^1 \times M$ such that the $\mathbb{Z}_r$-isotropy point set 
is diffeomorphic to $S^1 \times F$. In particular, if $S^1 \times F$ is not 
a symplectic manifold, the action of $G$ on $S^1 \times M$ is not symplectic.
\end{proposition}

\begin{proof}
The projection $S^1 \times M \to S^1$ yields a fibration
$S^1 \times_{\mathbb{Z}_r} M \to S^1/\mathbb{Z}_r$
with fiber $M$ and the gluing map $g \colon M \to M$, $x \mapsto gx$,
where $g$ is the generator of $\mathbb{Z}_r$. As $g$ is isotopic to
the identity on $M$, $S^1 \times_{\mathbb{Z}_r} M$ is
diffeomorphic to $S^1/\mathbb{Z}_r \times M \cong S^1 \times M$. The obvious action of $S^1$
on the twisted product $S^1 \times_{\mathbb{Z}_r} M$ is fixed point free. Moreover,
$$(S^1 \times_{\mathbb{Z}_r} M)_{\mathbb{Z}_r} \cong S^1/\mathbb{Z}_r \times M^{\mathbb{Z}_r}
   \cong S^1 \times F,$$
completing the proof.
\end{proof}

We wish to present a variation of Proposition~\ref{pro:twisted}.
Let $g$ be a periodic (with order $r$) diffeomorphism of a smooth manifold $M$.
Consider the mapping torus $T(g) = S^1\times_g M$, i.e., the fiber bundle
over $S^1$ with fiber $M$ and the gluing diffeomorphism $g$.

Since $g^r=\operatorname{id}$, $T(g) = S^1\times_{\mathbb{Z}_r} M$. Now,
$S^1$ acting on itself provides an action on $T(g)$. One can easily see
that this action is fixed point free and the $\mathbb{Z}_r$-isotropy point
set is diffeomorphic to $S^1/\mathbb{Z}_r \times F \cong S^1 \times F$,
where $F = M^{\mathbb{Z}_r}$.

\begin{proposition}\label{pro:torus}
Let $M$ be a closed symplectic manifold upon which $\mathbb{Z}_r$ acts smoothly
with non-empty fixed point set $F$, and for the generator $g$ of $\mathbb{Z}_r$,
the diffeomorphism
$$g \colon M \to M, \ x \mapsto gx,$$
is isotopic to a symplectomorphism of $M$. Then $S^1\times T(g)$ is a symplectic manifold
with a smooth fixed point free action of $S^1$ such that the $\mathbb{Z}_r$-isotropy point
set is diffeomorphic to $S^1 \times S^1 \times F$. In particular, if $S^1 \times S^1 \times F$
is not a symplectic manifold, the action of $G$ on $S^1\times T(g)$ is not symplectic.
\end{proposition}

\begin{proof}
As noted above, $S^1$ has a smooth fixed point free action on the mapping torus
$T(g)$ such that the $\mathbb{Z}_r$-isotropy point set is diffeommorphic to
$S^1 \times F$. Note that the diagonal action of $S^1$ on $S^1 \times T(g)$,
where $S^1$ acts trivially on the first factor $S^1$, is fixed point free
and the $\mathbb{Z}_r$-isotropy point set is diffeomorphic to $S^1 \times S^1 \times F$.

If $f \colon M \to M$ is a symplectomorphism isotopic to $g$, then $T(f)$
is diffeomorphic to $T(g)$, and $T(f)$ is a symplectic fibration over
a circle, i.e., it possesses a well-defined symplectic structure
{\emph{on fibers}}. This enables us to apply  Thurston's theorem
(see \cite{McDS}, chapter 6) to get a symplectic structure on
$S^1 \times T(f)$, a symplectic fibration over $S^1\times S^1$ with fibre $M$.
It suffices to check the claim that the cohomology class
of the symplectic form on $M$ is in the image of the cohomology homomorphism $i^{*},$
where $i \colon M \to T(f)$ is the inclusion.
The claim follows from the Mayer--Vietoris exact sequence resulting from
a decomposition of the base space of the fibration $T(f)$ into two intervals
(elements in cohomology which are invariant under the gluing map all are
in the image of $i^*)$, or from the Wang exact sequence.
\end{proof}

Propositions~\ref{pro:twisted} and \ref{pro:torus} allows us to obtain
smooth fixed point free actions of $S^1$ on symplectic manifolds with
non-symplectic $\mathbb{Z}_{pq}$-isotropy point sets.

For example, consider the action of $S^1$ on the symplectic manifold
$X^{2n+6}$ constructed in Theorem~\ref{thm:isotropy}, and restrict
the action of $S^1$ to $\mathbb{Z}_{pq}$. Then the diffeomorphism
$$g \colon X^{2n+6} \to X^{2n+6}, \ x \mapsto gx$$
given by the generator $g$ of $\mathbb{Z}_{pq}$ is isotopic to
the identity on $X^{2n+6}$, a symplectomorphism of $X^{2n+6}$.
Therefore, we may apply Proposition~\ref{pro:torus} to obtain
the required action of $S^1$ on the symplectic manifold
$S^1\times T(g) \cong S^1 \times S^1 \times X^{2n+6}$.

Also, we may replace $X^{2n+6}$ by $S^1 \times X^{2n+6}$
with the diagonal action of $\mathbb{Z}_{pq}$, where $\mathbb{Z}_{pq}$
acts trivially on $S^1$, and $\mathbb{Z}_{pq}$ acts on $X^{2n+6}$ by
restricting of the action of $S^1$ on $X^{2n+6}$. Then we may apply
Proposition~\ref{pro:twisted} for $M = S^1 \times X^{2n+6}$,
to obtain the required action of $S^1$ on the symplectic manifold
$S^1 \times M = S^1 \times S^1 \times X^{2n+6}$.

Propositions~\ref{pro:twisted} and \ref{pro:torus} are expected to be
useful also in the case where the manifold $M$ in question admits
a smooth action of $\mathbb{Z}_{pq}$ which does not extend to a smooth
action of $S^1$ on $M$. Examples of such actions of $\mathbb{Z}_{pq}$
on $M$ can be obtained from smooth actions of $\mathbb{Z}_r$ on disks,
whose fixed point sets are well-understood by the work of Oliver \cite{O}.

\medskip

\noindent
{\bf Acknowledgements.} This work was partially supported by the grant 1P03A 03330
of the Ministry of Science and Higher Education, Poland. The third author is
grateful to the Max-Planck-Institute for Mathematics in Bonn for hospitality
and excellent working conditions.

\medskip

\noindent {BH: \bf Mathematical Institute, Wroc\l aw University,

\noindent pl. Grunwaldzki 2/4,

\noindent 50-384 Wroc\l aw, Poland}

\medskip

\noindent

\medskip

\noindent {BH and AT: \bf Department of Mathematics and Information Technology,

\noindent University of Warmia and Mazury,

\noindent \.{Z}o{\l}nierska 14A, 10-561 Olsztyn, Poland}

\medskip

\noindent {KP:
\bf Faculty of Mathematics and Computer Science,

\noindent
Adam Mickiewicz University,

\noindent
Umultowska 87, 61-614 Pozna\'n, Poland}

\medskip

\begin{flushleft}
\tt hajduk@math.uni.wroc.pl

\tt kpa@amu.edu.pl

\tt tralle@matman.uwm.edu.pl
\end{flushleft}


\begin{thebibliography}{ABCDE}



\bibitem[A]{A} C. Allday, {Examples of circle actions on
 symplectic space}, Banach Center Publ. 45(1998), 87--90.

\bibitem[B]{B} S. Baldridge, {\it Seiberg-Witten vanishing theorem for $S^1$-manifolds
with fixed points}, Pacif. J. Math. 217(2004), 1--10.

\bibitem[Fa]{Fa} M. Farber, {\it Topology of closed one-forms}, Math. Surveys and Monogr.108, Amer.
Math. Soc., Providence, RI, 2004.

\bibitem[Fe]{Fe} K. Feldman, {\it Hirzebruch genera of manifolds supporting Hamiltonian circle actions}, Russian
Math. Surveys 56(2001), 978--979.

\bibitem[FS]{FS} R. Fintushel, R. Stern, {\it Knots, links
and 4-manifolds}, Invent. Math.{\bf 134}(1998), 363--400.

\bibitem[FV]{FV} S. Friedl, S. Vidussi, {\it Symplectic $S^1\times N^3$, subgroup separability,
and vanishing Thurston norm}, J. Amer. Math. Soc. 21(2008), 597--610.



\bibitem[GuSt]{GuSt} V. Guillemin, S. Sternberg, {\it Symplectic techniques in
physics.} Cambridge Univ. Press, 1984.

\bibitem[McDS]{McDS} D. McDuff, S. Salamon, {\it Introduction
to symplectic topology}, Oxford Univ. Press, 1998.

\bibitem[O]{O} B. Oliver, {\it Fixed point sets and tangent bundles of actions on disks and
                            Euclidean spaces}, Topology 35(1996), 583--615.

\bibitem[Oz]{Oz} Y. Ozan, {\it On cohomology of invariant submanifolds of Hamiltonian actions},
               Michigan Math. J. 53(2005), 579--584.

\bibitem[P]{P} J. Park, {\it Exotic smooth structures on $4$-manifolds},
               Forum. Math. 14(2002), 915--929.

\bibitem[Pa]{Pa} K. Pawa{\l}owski,
{\it Smooth circle actions on highly symmetric manifolds},
     Math. Ann. 341(2008), 845--858.

\bibitem[Pu]{Pu} V. Puppe, {\it Do manifolds have little symmetry?}, J. Fixed Point Theory Appl. 2(2007), 85--96.

\bibitem[Po]{Po} L. Polterovich, {\it Growth of maps, distortion in groups and
symplectic geometry}, Invent. Math. 150(2002), 655--686.

\bibitem[S]{S} A. Scorpan, {\it The Wild World of $4$-Manifolds. } AMS. Providence, RI, 2005.

\bibitem[T]{T} C.H. Taubes, {\it The Seiberg--Witten invariants and symplectic forms},
                Math. Res. Letters 2 (1994), 9--14.

\bibitem[W]{W} C.T.C. Wall, {\it Classification problems in differential topology V. On certain 6-manifolds}, Invent. Math. 1(1966), 355-374. Acad. Press, New York, 1970.

\end{thebibliography}
\end{document}